\documentclass[11pt]{amsart}
\usepackage[latin1]{inputenc}
\usepackage{latexsym}
\usepackage{color,graphicx,shortvrb}
\usepackage{amsfonts}
\usepackage{amsmath,amssymb,amsthm}
\usepackage[colorlinks, bookmarks=true]{hyperref}
\usepackage[active]{srcltx}

\theoremstyle{plain}
\newtheorem{theorem}{Theorem}[section]

\newtheorem{lemma}[theorem]{Lemma}
\newtheorem{proposition}[theorem]{Proposition}

\theoremstyle{definition}

\renewcommand{\leq}{\leqslant}
\renewcommand{\geq}{\geqslant}\usepackage{amssymb}

\newcommand{\vr}{\varepsilon}

\newcommand{\be}{\begin{equation}}
\newcommand{\ee}{\end{equation}}

\newcommand{\N}{\mathbb N}

\newcommand{\pp}{\ensuremath{(\mathcal{P})}}

\newcommand{\hra}{\hookrightarrow}

\newcommand{\ball}{{\mathrm{B}}}
\newcommand{\as}{a^{(\Pi_p)}}  
\newcommand{\at}{a^{(\Pi_2)}}
\newcommand{\cs}{c^{(\Pi_p)}}
\newcommand{\asa}{a^{(\mathcal{A})}}
\newcommand{\csa}{c^{(\mathcal{A})}}
\newcommand{\ass}{a^{(\Pi_{tr})}}
\newcommand{\css}{c^{(\Pi_{tr})}}

\newcommand{\A}{{\mathcal{A}}}

\renewcommand{\span}{\mathrm{span}}

\def \dim{{\mathrm{dim}} \, }
\def \rank{{\mathrm{rank}} \, }
\def \ker{{\mathrm{ker}} \, }

\def \vr{\varepsilon}

\def \ran{{\mathrm{ran}} \, }
\def \codim{{\mathrm{codim}} \,}
\def \diag{{\mathrm{diag}} \, }

\def \eqalign#1{\null\,\vcenter{\openup\jot 
   \ialign{\strut\hfil$\displaystyle{##}$&$
      \displaystyle{{}##}$\hfil \crcr#1\crcr}}\,}

\begin{document}

\baselineskip=15.5pt

\numberwithin{equation}{section}

\pagestyle{headings}

\title[Rate of decay of $s$-numbers]{Rate of decay of $s$-numbers}

\author[T.~Oikhberg]{T.~Oikhberg}
\address{
Department of Mathematics, University of California - Irvine, Irvine CA 92697,
and
Department of Mathematics, University of Illinois at Urbana-Champaign, Urbana, IL 61801}

\email{toikhber@math.uci.edu}



\begin{abstract}
For an operator $T \in B(X,Y)$, we denote by $a_m(T)$, $c_m(T)$, $d_m(T)$, and
$t_m(T)$ its approximation, Gelfand, Kolmogorov, and absolute numbers.
We show that, for any infinite dimensional Banach spaces $X$ and $Y$, and
any sequence $\alpha_m \searrow 0$, there exists $T \in B(X,Y)$ for which
the inequality
$$
3 \alpha_{\lceil m/6 \rceil} \geq a_m(T) \geq \max\{c_m(t), d_m(T)\} \geq
\min\{c_m(t), d_m(T)\} \geq t_m(T) \geq \alpha_m/9
$$
holds for every $m \in \N$. Similar results are obtained for other $s$-scales.
\end{abstract}

\maketitle


\section{Introduction and main results}

In this paper, we investigate the existence of an operator $T \in B(X,Y)$
($X$ and $Y$ are infinite dimensional Banach spaces) whose sequence of
$s$-numbers $(s_n(T))$ ``behaves like'' a prescribed sequence $(\alpha_n)$.

For a linear operator $T$ between Banach spaces $X$ and $Y$,
define its {\it approximation numbers} $a_n$, {\it Kolmogorov numbers} $d_n$,
{\it Gelfand numbers} $c_n$, {\it symmetrized} (or {\it absolute}) {\it numbers}
$t_n$, {\it Weyl numbers} $x_n$, {\it Chang numbers} $y_n$, and {\it Hilbert
numbers} $h_n$:
\begin{equation}
\begin{array}{lll}
a_n(T)
& = &
\inf \{ \|T-S\| : S \in B(X,Y), \, \rank S < n \} , 
\cr
d_n(T) 
& = &
\inf \{ \|q T\| : q : Y \to Y/F \, {\mathrm{quotient~map}}, \, \dim F < n \}
\cr
& = &
\inf \{ a_n(T q) : q : \tilde{X} \to X \, {\mathrm{quotient~map}} \} ,
\cr
c_n(T)
& = &
\inf \{ \|T|_E\| : E \hra X, \, \codim E < n \}
\cr
& = &
\inf \{ a_n(j T) : j : Y \to \tilde{Y} \, {\mathrm{isometry}} \} ,
\cr
t_n(T)
& = &
\inf \{ a_n(j T q) : q : \tilde{X} \to X \, {\mathrm{quotient~map}}, \,
   j : Y \to \tilde{Y} \, {\mathrm{isometry}} \} ,
\cr
x_n(T)
& = &
\inf \{ a_n(T u) : u : \ell_2 \to X , \, \|u\| \leq 1 \} ,
\cr
y_n(T)
& = &
\inf \{ a_n(v T) : v : Y \to \ell_2 , \, \|v\| \leq 1 \} ,
\cr
h_n(T)
& = &
\inf \{ a_n(v T u) : u : \ell_2 \to X , \, v : Y \to \ell_2 ,  \, \|u\| \|v\| \leq 1 \} .
\cr
\end{array}
\label{eq:s_numbers}
\end{equation}
We refer the reader to \cite{CS, PieS} for general information about these and
other $s$-numbers. Note that
$t_n(T) \leq \min\{c_n(T), d_n(T)\} \leq \max\{c_n(T), d_n(T)\} \leq a_n(T)$
for any operator $T$.
We say that an operator $T$ is {\it approximable} if $\lim_n a_n(T) = 0$.
It is well known that $T$ is compact if and only if $\lim_n d_n(T) = 0$ if and only if $\lim_n c_n(T) = 0$.
Any approximable operator is compact, but the converse is not true, due to the
existence of Banach spaces failing the Approximation Property.

Throughout the paper, the notation $\alpha_i \searrow 0$ means that
the sequence $(\alpha_i)$ satisfies $\alpha_1 \geq \alpha_2 \geq \ldots \geq 0$,
and $\lim_i \alpha_i = 0$.

We are motivated by Bernstein's Lethargy Theorem, stating
that, for any Banach space $X$, any strictly increasing chain of finite dimensional subspaces
$X_1 \hra X_2 \hra \ldots \hra X$, and any sequence $\alpha_i \searrow 0$,
there exists $x \in X$ such that $d(x,X_i) = \alpha_i$
for every $i$ (for the proof, see e.g.~\cite[Section II.5.3]{singer}).
This theorem was later generalized
to the more general class of $FS$-spaces \cite{Mi}.
Certain partial results for chains $X_1 \hra X_2 \hra \ldots \hra X$ of
infinite dimensional subspaces of a Banach space $X$ can be found in
\cite[Section I.6.3]{singer}. Related results were obtained for
general approximation schemes in \cite{AO}.


In a similar vein, one can study the existence
of operators whose sequences of $s$-numbers behave in a prescribed fashion.
First results of this kind were obtained in \cite{HMRe}.
Among other things, it was proved that, of every pair of infinite dimensional
Banach spaces $(X, Y)$, and any $\vr > 0$, there exist infinite dimensional
$X_0 \hra X$ and $Y_0 \hra Y$, such that for any sequence $\alpha_i \searrow 0$
there exists $T \in B(X_0, Y_0)$ with the property that
$\alpha_i \leq a_i(T) \leq (1+\vr) \alpha_i$ for every $i$. Furthermore, for many
pairs $(X,Y)$, the existence of $T \in B(X,Y)$ satisfying
$\alpha_i \leq a_i(T) \leq M \alpha_i$ ($M$ is a constant, depending on $(X,Y)$)
is demonstrated. These results were sharpened in \cite{ALe}, where it was shown
that, for a certain class of pairs $(X,Y)$, for any $\alpha_i \searrow 0$
there exists $T \in B(X,Y)$ such that $a_i(T) = \alpha_i$ for every $i$.
One should also mention \cite{KRe}, where operators with prescribed
eigenvalue sequences are constructed.

The main result of this paper is:

\begin{theorem}\label{thm:controlled}
Suppose $X$ and $Y$ are infinite dimensional Banach spaces, and $\alpha_k \searrow 0$.
Then there exists an approximable $T : X \to Y$ such that $\|T\| \leq 2 \alpha_1$,
and, for every $m$,
$3 \alpha_{\lceil m/6 \rceil} \geq a_m(T) \geq t_m(T) \geq \alpha_m/9$,
$\min\{x_m(T), y_m(T)\} \geq \alpha_m/(9\sqrt{m})$,
and $h_m(T) \geq \alpha_m/(9m)$.
\end{theorem}

In general, one cannot omit the condition $\lim \alpha_m = 0$.
Indeed, suppose $X = \ell_p$ (or $X = c_0$), and $Y = \ell_q$, with $p > q \geq 1$
($\infty > q \geq 1$ if $X = c_0$).
By Pitt's theorem \cite[Proposition 2.c.3]{LT1}, any $T \in B(X,Y)$ is compact.
Furthermore, $Y$ has the Approximation Property, hence, by \cite[Theorem 1.e.4]{LT1},
any compact operator into $Y$ is approximable. Thus,
$\lim_m a_m(T) = 0$ for any $T \in B(X,Y)$.

The lower estimates for $x_m(T)$, $y_m(T)$, and $h_m(T)$ are best possible, too.

\begin{proposition}\label{prop:optimal}
$\lim \sqrt{k} x_k(T) = \lim \sqrt{k} y_k(T) = \lim k h_k(T) = 0$
for any $T \in B(c_0, \ell_1)$.
\end{proposition}

Note that, unlike the results of \cite{ALe, HMRe}, Theorem~\ref{thm:controlled}
covers all pairs $(X,Y)$ of infinite dimensional Banach spaces. We do not
know whether this theorem can be strengthened to obtain $T \in B(X,Y)$
with (say) $\alpha_i \leq a_i(T) \leq C \alpha_i$, for some fixed constant $C$.
However, for some pairs $(X,Y)$, one cannot find an operator $T : X \to Y$ with
{\it precisely} the prescribed Gelfand or approximation numbers.
Recall that a Banach space $X$ is called {\it strictly convex} if for every
$x, y \in X$, $\|x+y\| = \|x\| + \|y\|$ can hold only if $x$ and $y$ are
scalar multiples of each other (see e.g. \cite{JoL}). Therefore,
any $x^* \in X^*$ can attain its norm at no more than one point of the
unit ball of $X$. It is known that for every separable Banach space there exists
an equivalent strictly convex norm (and more -- see Section 1 of \cite{God}).

\begin{proposition}\label{prop:second}
Suppose $X$ is a strictly convex reflexive Banach space, and
$T : X \to c_0$ is compact. Then $a_2(T) = c_2(T) < a_1(T) = c_1(T) = \|T\|$.
\end{proposition}

The condition that $T$ is compact (equivalently, $\lim_k c_k(T) = 0$) is
essential: if $T$ is the formal embedding of $\ell_p$ to $c_0$
($1 \leq p < \infty$), then $c_k(T) = 1$ for each $k$.
However, the compactness of $T$ is equivalent to $\lim a_i(T) = 0$.

Note also that $\ell_1$ is not strictly convex, hence the above proposition
doesn't apply to the operators from $\ell_1$ to $c_0$.
In fact, \cite{ALe} shows that, for any decreasing sequence $(\alpha_m)$
with $\lim \alpha_m = 0$, there exists $T \in B(\ell_1, c_0)$ such that
$a_m(T) = \alpha_m$ for every $m$.

Now suppose ${\mathcal{A}}$ is a quasi-Banach operator ideal,
equipped with the norm $\| \cdot \|_{\mathcal{A}}$ (see e.g.~\cite{DJT, PieID, T-J}
for the definition and basic properties of operator ideals).
Define the {\it ${\mathcal{A}}$-approximation numbers} by setting
$$
a^{(\mathcal{A})}_n(T) = \inf_{u \in B(X,Y), \rank u < n} \|T-u\|_{\mathcal{A}} .
$$
The {\it ${\mathcal{A}}$-Gelfand numbers} 
are defined by
$$
c^{(\mathcal{A})}_n(T) = \inf_{E \hra X, \codim E < n} \|T|_E\|_{\mathcal{A}} .
$$

We are especially interested in the ideals of $p$-factorable
and $(t,r)$-summing operators.
Recall that $T \in B(X,Y)$ is called {\it $p$-factorable} ($1 \leq p \leq \infty$)
if it can be represented as $T = T_2 T_1$, with $T_1 \in B(X,L_p(\mu))$ and
$T_2 \in B(L_p(\mu), Y)$. The associated norm is given by
$\gamma_p(T) = \inf \|T_2\| \|T_1\|$, with the infimum running over all
representations of the above form. The ideal of all $p$-summing operators
is denoted by $\Gamma_p$.

An operator $T \in B(X,Y)$ is {\it $(t,r)$-summing} ($1 \leq r \leq t \leq \infty$)
if there exists a constant $C$ such that, for every $x_1, \ldots, x_n \in X$,
$$
\Big( \sum_{i=1}^n \|T x_i\|^t \Big)^{1/t} \leq
C \Big( \sup_{x^* \in X^*, \|x^*\| \leq 1} \sum_{i=1}^n |\langle x^*, x_i \rangle|^r \Big)^{1/r} .
$$
The infimum of all $C > 0$ with the above property is denoted by $\pi_{t,r}(T)$,
and the corresponding ideal -- by $\Pi_{tr}$. When $t=r$, we use the notation
$\pi_r$ and $\Pi_r$, and the term {\it $r$-summing}.

For $p$-factorable operators, we have:

\begin{theorem}\label{cor:factor}
For $1 < p < \infty$, there exists a constant $K_p$ such that, whenever
$X$ and $Y$ are infinite dimensional Banach spaces, and $\alpha_k \searrow 0$,
there exists an approximable $T : X \to Y$ such that $\|T\| \leq 2 \alpha_1$
and, for every $m$,
$$
K_p \alpha_{\lceil m/6 \rceil} \geq a_m^{(\Gamma_p)} (T) \geq
c_m^{(\Gamma_p)}(T) \geq \alpha_m/9 .
$$
\end{theorem}

As we shall see below, the operator $T$ constructed in Theorem~\ref{thm:controlled} has
the properties described by this theorem.

Next we handle the ideal of $p$-summing operators $\Pi_p$.

\begin{theorem}\label{thm:2sum}
If $X$ and $Y$ are infinite dimensional Banach spaces, and $\alpha_k \searrow 0$,
there exists a $2$-summing map $T \in B(X,Y)$, such that
$$
c_p \alpha_{18m} \leq \cs_m(T) \leq \as_m(T) 
\leq 3 \alpha_{\lceil 4m/5 \rceil}
$$
for every $m$, and every $p \in [2,\infty)$
($c_p$ is a constant depending on $p$).
\end{theorem}

For certain pairs $(X,Y)$, one can construct $T \in B(X,Y)$ with the
prescribed rate of decay of $(\csa_m(T))$ and $(\asa_m(T))$ for other
classes of ideals $\mathcal{A}$. Recall that a Banach operator ideal
${\mathcal{A}}$ is called {\it $1$-injective} if, for any $u \in B(X,Y)$,
and any isometric injection $J : Y \to Y_0$, we have
$\|u\|_{\mathcal{A}} = \|J u\|_{\mathcal{A}}$. For instance, the ideal
$\Pi_{tr}$ of $(t,r)$-summing operators is $1$-injective.

\begin{theorem}\label{thm:no_cotype}
Suppose $\alpha_k \searrow 0$,
and the Banach spaces $X$ and $Y$ have no non-trivial cotype and no non-trivial
type, respectively. Then there exists $T \in B(X,Y)$ such that
$$
\frac{1}{50} \alpha_{18m}
\leq c_m(T) \leq \csa_m(T) \leq \asa_m(T) \leq 4 \alpha_{\lceil 4m/5 \rceil}
$$
for every $m$, and every $1$-injective Banach operator ideal $\A$.
\end{theorem}

We shall say that a Banach space $X$ has {\it Property $\pp_C$} ($C \geq 1$) if, for any
$n \in \N$, and any finite codimensional $X^\prime \hra X$, there exists
an $n$-dimensional $E \hra X^\prime$ such that $d(E, \ell_2^n) \leq C$, and
$E$ is $C$-complemented in $X$. By \cite{PiHOLO}, any space with non-trivial
type has Property $\pp_C$, for some $C$. Consequently, any Banach space containing
a complemented subspace of non-trivial type has Property $\pp_C$ for some $C$.

\begin{theorem}\label{thm:type}
Suppose an infinite dimensional Banach space $X$ has Property $\pp_C$, for some $C > 1$.
Then, for any infinite dimensional Banach space $Y$, and any sequence
$\alpha_k \searrow 0$, there exists $T \in B(X,Y)$ such that
$$
\frac{1}{20C^2} \alpha_{18m} \leq \css_m(T) \leq \ass_m(T) \leq
4 \alpha_{\lceil 4m/5 \rceil}
$$
for every $m  \in \N$, and for any $t$ and $r$ satisfying $1 \leq r \leq \min\{2,t\}$,
and $1/r - 1/t < 1/2$.
\end{theorem}

We do not know how well one can control the rate of decay of $(a^{(\mathcal{A})}_i(T))$
for general ideals $\mathcal{A}$. It was shown in \cite{AO} that, for any quasi-Banach
(respectively, Banach) ideal ${\mathcal{A}}$, and every sequence $\alpha_i \searrow 0$,
there exists $T \in B(X,Y)$ such that $\lim a^{(\mathcal{A})}_i(T) = 0$, and $a^{(\mathcal{A})}_i(T) \geq \alpha_i$ for infinitely many (respectively, all)
values of $i$.

We prove the results stated above in Section~\ref{proofs}.
Throughout, we assume $\alpha_1 > 0$ (the case of $\alpha_1 = 0$
is trivial). We use the common Banach and operator space notation
(see e.g.~\cite{DJT, LT1}. $\ball(X)$ denotes the closed unit ball of $X$.
$d(E,F)$ stands for the {\it Banach-Mazur distance} between Banach spaces
$E$ and $F$. That is, $d(E,F) = \inf \|u\| \|u^{-1}\|$, with the infimum running over all
invertible maps $u : E \to F$.


\section{Proofs}\label{proofs}

The proofs of some of the results requires using copies of $\ell_2^n$
as building blocks (in a way reminiscent of \cite{PieSMALL}).
We thus need:

\begin{lemma}\label{l:dvor}
Suppose $c > 1$, $(n_k)$ is a sequence of positive integers, and
$\Gamma$ is an infinite set.
\begin{enumerate}
\item
Suppose $J$ is an isometric injection of an infinite dimensional
Banach space $Y$ to $\ell_\infty(\Gamma)$.
Then there exist subspaces $(F_k)$ of $Y$, and finite rank maps
$v_k \in B(\ell_\infty(\Gamma))$, such that:
$(i)$ $d(F_k, \ell_2^{n_k}) < \sqrt{c}$,
$(ii)$ $\|v_k\| < c+1$,
$(iii)$ $v_k J j_k = J j_k$,
$(iv)$ for $s \neq k$, $v_s J j_k = 0$
($j_k$ denotes the canonical inclusion of $F_k$ into $Y$).
%
\item
Suppose $X$ is an infinite dimensional Banach space, and
$Q : \ell_1(\Gamma) \to X$ is a quotient map. Then there exist
quotients $E_k$ of $X$ ($q_k : X \to E_k$ is the quotient map)
and weak$^*$ continuous maps $u_k \in B(\ell_\infty(\Gamma))$,
such that
$(i)$ $d(E_k, \ell_2^{n_k}) < \sqrt{c}$,
$(ii)$ $\|u_k\| < c+1$,
$(iii)$ $u_k|_{Q^* q_k^* (E_k^*)} = I_{Q^* q_k^* (E_k^*)}$, and
$(iv)$ for $s \neq k$, $u_s Q^* q_k^* = 0$.
\end{enumerate}
%
\end{lemma}

\begin{proof}
%
(1)
Select $\lambda \in (1, \sqrt{c})$ in such a way that $\lambda(1+\lambda) < 1+c$.
We construct the spaces $F_k$ and operators $v_k$ recursively.
By Dvoretzky's Theorem, there exists $F_1 \hra Y$ such that $d(F_1, \ell_2^{n_1}) < \sqrt{c}$.
Furthermore, we can find a finite rank projection
$v_1 \in B(\ell_\infty(\Gamma))$ such that $v_1|_{F_1} = I_{F_1}$, and $\|v_1\| < c$ .

Now suppose $F_1, \ldots, F_{k-1}, v_1, \ldots, v_{k-1}$ with the desired properties
have been constructed. Find a finite rank projection $P_1 \in B(\ell_\infty(\Gamma))$
such that $\|P_1\| < \lambda$, and $P_1 v_s = v_s$ for $1 \leq s < k$.
Find $F_k \hra Y \cap \ker P_1 \cap \big( \cap_{s=1}^{k-1} \ker v_s \big)$, such that
$d(F_k, \ell_2^{n_k}) < \sqrt{c}$. Finally, find a finite rank projection
$P_2 \in B(\ell_\infty(\Gamma))$ such that $\|P_2\| < \lambda$, and $P_2|_{F_k} = I_{F_k}$.
Let $v_k = P_2 (I-P_1)$. Then $v_k|_{F_k} = I_{F_k}$. By our choice of $\lambda$,
$\|v_k\| < 1 + c$.
For $s < k$, we have $v_k v_s = 0$, and $v_s|_{F_k} = 0$.
Thus, $v_s J j_k = 0$ if $s \neq k$.

(2)
By (1), there exist subspaces $G_1, G_2, \ldots$ of $X^*$, and
finite rank operators $w_k \in B(\ell_\infty(\Gamma))$, such that
$d(G_k, \ell_2^{n_k}) < \sqrt{c}$, $\|w_k\| < c+1$, $w_k|_{G_k} = I_{G_k}$,
$w_k w_s = 0$ for $s < k$, and $w_s|_{G_k} = 0$ for $s > k$.
By \cite[Theorem 2.5]{OP}, there exists a sequence of finite rank maps
$u_k \in B(\ell_1(\Gamma))$ such that $\|u_k\| < c+1$, $\ran u_k^* = \ran w_k$,
and $(u_k^* - w_k)|_{G_k \cup (\cup_{s<k} \ran w_s)} = 0$.
Furthermore, the isometric embedding $i_k : G_k \to X^*$ is the
dual of the quotient map $q_k : X \to E_k = G_k^*$, where
$q_k = i_k^*|_X$.
\end{proof}

\begin{proof}[Proof of Theorem~\ref{thm:controlled}]
Select a set $\Gamma$ for which there exist an isometric embedding
$J : Y \to \ell_\infty(\Gamma)$, and a quotient map $Q : \ell_1(\Gamma) \to X$.
Set $n_0 = 0$, and find a sequence $(n_k) \subset \N$ such that, for each $k$,
(i) $n_k > 5(n_{k-1} + 1)$, and (ii) $\alpha_{n_k} \leq \alpha_{n_{k-1}+1}/5$.
Select $c > 1$ such that $c^2(1+c)^2 < 9/2$. By Lemma~\ref{l:dvor},
there exist embeddings $j_k : F_k \to Y$, quotient maps $q_k : X \to E_k$, and
finite rank operators
$u_k, v_k \in B(\ell_\infty(\Gamma))$, such that, for each $k$:
\begin{itemize}
\item
$\max\{ d(E_k, \ell_2^{n_k}), d(F_k, \ell_2^{n_k}) \} < \sqrt{c}$.
\item
$\max\{\|u_k\|, \|v_k\|\} < c+1$.
\item
$u_k$ is weak$^*$ continuous (hence $u_k^*$ maps $\ell_1(\Gamma) = \ell_\infty(\Gamma)_*$
into itself).
\item
$u_k Q^* q_k^* = Q^* q_k^*$ (equivalently, $q_k Q u_k^* = q_k Q$),
and $v_k J j_k = J j_k$.
\item
For $s \neq k$, $u_s Q^* q_k^* = 0$, and $v_s J j_k = 0$.
\end{itemize}
For each $k$, find contractions $U_k : E_k \to \ell_2^{n_k}$ and
$V_k : \ell_2^{n_k} \to F_k$, such that their inverses have norms
smaller than $\sqrt{c}$. For $1 \leq j \leq n_k$, set
$\beta_{jk} = \min\{\alpha_{n_{k-1}+1}, \alpha_j\}$.
Denote the canonical basis in $\ell_2^{n_k}$ by $(\delta_{jk})_{j=1}^{n_k}$,
and define the diagonal operator $D_k \in B(\ell_2^{n_k})$ by setting
$D_k \delta_{jk} = \beta_{jk} \delta_{jk}$ ($1 \leq j \leq n_k$).
Let $S_k = V_k D_k U_k$. Then $\|S_k\| \leq \alpha_{n_{s-1}+1}$,
hence $T = \sum_{s=1}^\infty j_s S_s q_s$ is approximable, and
$\|T\| < 2 \alpha_1$.

To estimate $t_m(T) = a_m(J T Q)$ from below, find $k$ satisfying
$n_{k-1} < m \leq n_k$. Recall that $v_k J j_k = J j_k$, $q_k Q u_k^* = q_k Q$,
and, for $s \neq k$, $v_k J j_s$ and $q_s Q u_k^*$ vanish. Therefore,
for any $s$-scale,
\begin{equation}
(1+c)^2 s_m(J T Q)
\geq
s_m(v_k J T Q u_k^*) =
s_m ( \sum_s v_k J j_s S_s q_s Q u_k^* )
=
s_m(J j_k S_k q_k Q) .
\label{eq:test}
\end{equation}
Consequently,
\begin{equation}
t_m(T) = a_m(J T Q) \geq (1+c)^{-2} a_m(J j_k S_k q_k Q) .
\label{eq:t_m}
\end{equation}
To proceed further, note that $q_k Q(\ball(\ell_1(\Gamma))) = \ball(E_k)$, and
$J j_k|_{\ran S_k} = I_{\ran S_k}$. Let $G = V_k(\span[\delta_{jk} : j \leq m])$.
Then $J j_k S_k q_k Q(\ball(\ell_1(\Gamma)))$
contains $c^{-1} \alpha_m \ball(G)$, and therefore (see e.g.
Lemma 1.19 of \cite{HMVZ}),
$$
a_m(J j_k S_k q_k Q) \geq d_m(J j_k S_k q_k Q) \geq c^{-1} \alpha_m .
$$
Together with \eqref{eq:t_m}, this yields the desired estimate for $t_m(T)$.

Next we estimate $a_m(T)$ from above. Let $n^\prime_k = n_1 + \ldots + n_{k-1}$
(by our assumption on the sequence $(n_j)$, $n^\prime_k < 4 n_{k-1}/3$).
Assume first that $m > 3 n_k^\prime/2$. Write $T = T^{(1)} + T^{(2)} + T^{(3)}$,
where 
$$
T^{(1)} = \sum_{s=1}^{k-1} J j_s S_s q_s Q , \, \, \,
T^{(2)} = J j_k S_k q_k Q , \, \, \, {\mathrm{and}} \, \, \,
T^{(3)} = \sum_{s=k+1}^\infty J j_s S_s q_s Q .
$$
Then
$$
a_m(T) \leq a_m(T^{(1)} + T^{(2)}) + \|T^{(3)}\| \leq
a_{m - \rank T^{(1)}}(T^{(2)}) + \|T^{(3)}\| .
$$
Then $\|T^{(3)}\| \leq \sum_{s=k}^\infty \|S_{s+1}\|
\leq \sum_{s=k}^\infty \alpha_{n_s+1}$. But, for $j \geq 0$,
$\alpha_{n_{k+j}+1} \leq 5^{-j} \alpha_{n_{k}} \leq 5^{-j} \alpha_m$,
hence $\|T^{(3)}\| \leq 5 \alpha_m/4$.
Furthermore, $\rank T^{(1)} \leq n^\prime_k$.
Therefore, by \cite{PieS},
$$
a_{m - n^\prime_k}(T^{(2)}) \leq a_{m - n^\prime_k}(D_k) \leq
\beta_{m - n^\prime_k,k} = \alpha_{\max\{m - n^\prime_k, n_{k-1}+1\}} \leq
\alpha_{\lceil m/3 \rceil} ,
$$
hence $a_m(T) \leq 3 \alpha_{\lceil m/3 \rceil}$.

Now suppose $n_{k-1} < m \leq 3 n_k^\prime/2$. As $n_{k-1} > n_{k-1}^\prime$, the
reasoning above shows $a_m(T) \leq a_{n_{k-1}}(T) \leq 3 \alpha_{\lceil n_{k-1}/3 \rceil}$.
Furthermore, $m \leq 3 n_k^\prime/2 < 2 n_{k-1}$, hence
$a_m(T) \leq 3 \alpha_{\lceil m/6 \rceil}$.

Before establishing lower estimates for other $s$-numbers mentioned in the theorem,
recall a few known facts.
By \cite{KKo}, any linear operator $u : Z_0 \to G$ ($G$ is a finite dimensional
space with $\dim G > 1$, and $Z_0$ is a subspace of a Banach space $Z$) has
an extension $\tilde{u} : Z \to G$, satisfying $\|\tilde{u}\| < \sqrt{\dim G} \|u\|$.
Moreover, by \cite{OP}, $\tilde{u}$ can be taken to weak$^*$ continuous
if $Z_0$ is finite dimensional.


To estimate $x_m(T)$, pick $k$ with $n_{k-1} < m \leq n_k$.
We consider the case of $m > 1$, as $x_1(T)$ can be estimated similarly.
By \eqref{eq:test}, $x_m(T) \geq (1+c)^{-2} x_m(J j_k S_k q_k Q)$.
Let $H = \span[\delta_{jk} : 1 \leq k \leq m]$, and
$E = U_k^{-1}(H)$. We find a contraction $a : \ell_2^m \to \ell_1(\Gamma)$,
for which
\begin{equation}
q_k Q a(\ball(\ell_2^m)) \subset (cm)^{-1/2} \ball(E) .
\label{eq:contains}
\end{equation}
Once we have such an $a$, recall that $\|D_k \xi\| \geq \alpha_m \|\xi\|$
for any $\xi \in H$. Therefore,
$$
J j_k S_k q_k Q a (\ball(\ell_2^m)) \supset
c^{-3/2} m^{-1/2} \alpha_m \ball(V_k(H)) ,
$$
hence
$$
x_m(J j_k S_k q_k Q) \geq
d_m(J j_k S_k q_k Q a) \geq c^{-3/2} m^{-1/2} \alpha_m .
$$

To construct $a$ as above, denote the inclusion of $E$ into $E_k$ by $i$.
Recall that $Q^* q_k^*$ is an isometric embedding of $E_k^*$
into $\ell_\infty(\Gamma)$. As noted above, $i^* : E_k^* \to E^*$ has a
weak$^*$ continuous extension $a_0^* : \ell_\infty(\Gamma) \to E^*$, such
that $\|a_0\| < \sqrt{m}$. Then $a = (cm)^{-1/2} a_0 U_k^{-1}|_H$
satisfies \eqref{eq:contains}.

To handle $y_m(T)$ and $h_m(T)$, we need a contraction
$b : \ell_\infty(\Gamma) \to \ell_2^m$
such that $\sqrt{cm} \, b J j_k = V_k^{-1}|_{V_k(H)}$
(here, we identify $H$ with $\ell_2^m$).
To show that such a $b$ exists, note that the operator
$V_k^{-1} : V_k(H) \to H$ has an extension $b_0 : \ell_\infty(\Gamma) \to H$,
with $\|b_0\| < \sqrt{cm}$. Then $b = (cm)^{-1/2} b_0$ has the desired properties.

By \eqref{eq:test}, $y_m(T) \geq (1+c)^{-2} a_m(b J j_k S_k q_k Q)$.
Recall that $\|D_k \xi\| \geq \alpha_m \|\xi\|$ for any $\xi \in H$. Thus,
$b J j_k S_k q_k Q (\ball(\ell_1(\Gamma)) \supset
c^{-3/2} m^{-1/2} \alpha_m \ball(V_k(H))$. By \eqref{eq:test},
$$
y_m(T) \geq (1+c)^{-2} a_m(b J j_k S_k q_k Q) \geq
c^{-3/2} (1+c)^{-2} m^{-1/2} \alpha_m .
$$
Furthermore, $b J j_k S_k q_k Q a(\ball(\ell_2^m))$ contains
$c^{-2} m^{-1} \alpha_m \ball(\ell_2^m)$, hence
$$
h_m(T) \geq (1+c)^{-2} a_m(b J j_k S_k q_k Q a) \geq
m^{-1} c^{-2} (1+c)^{-2} \alpha_m .
$$
As $c^2 (1+c)^2 < 9/2$, we are done.
\end{proof}

\begin{proof}[Proof of Proposition~\ref{prop:optimal}]
Denote by $P_N$ the projection onto the span
of the first $N$ elements of the canonical basis in $c_0$. Then $P_N^*$ is the projection
onto the span of the first $N$ elements of the canonical basis in $\ell_1$.
First show that, for $T \in B(c_0, \ell_1)$, 
\begin{equation}
\lim_N \|T - P_N^* T P_N\| = 0
\label{eq:truncate}
\end{equation}
As noted in the paragraph preceding the statement of this theorem,
for every $\vr > 0$ there exists a finite rank
operator $S$ satisfying $\|T - S\| < \vr/3$. Write $S = \sum_{i=1}^n y_i \otimes z_i$,
with $y_i, z_i \in \ell_1$ (that is, for $x \in c_0$,
$Sx = \sum_{i=1}^n \langle y_i , x \rangle z_i$).
Then $P_N^* S P_N x = \sum_{i=1}^n \langle P_N^* y_i , x \rangle P_N^* z_i$.
Note that $\lim_N P_N^* y = y$ for any $y \in \ell_1$, hence
there exists $M \in \N$ with $\|P_N^* S P_N - S\| < \vr/3$ for any $N \geq M$.
For such values of $N$,
$$
\|T - P_N^* T P_N\| \leq \|T - S\| + \|S - P_N^* S P_N\| + \|P_N^* (S - T) P_N\| < \vr .
$$
As $\vr$ is arbitrary, \eqref{eq:truncate} follows.

Recall that, for any $s$-scale, $s_m(u+v) \geq s_{m - \rank v}(u)$ if $v$
is a finite rank operator, and $m \geq \rank v$. In the above notation,
$\rank (P_N^* T P_N) \leq N$, hence
\begin{equation}
s_m(T) \leq s_{m - N}(T - P_N^* T P_N) .
\label{eq:s_m}
\end{equation}

We study $(x_k(\cdot))$ first. By Grothendieck Theorem,
$\pi_2(u) \leq K_G \|u\|$ for any $u \in B(c_0, \ell_1)$. Furthermore,
for any operator $u$ and $k \in \N$, $x_k(u) \leq \pi_2(u)/\sqrt{k}$
\cite[Lemma 9]{Ko}. Thus, for any $u \in B(c_0, \ell_1)$ and $k \in \N$,
\begin{equation}
x_k(u) \leq K_G \|u\|/\sqrt{k} .
\label{eq:x_k}
\end{equation}

Now fix $T \in B(c_0, \ell_1$). For any $\vr > 0$, there exists $M \in \N$ such that
$\|T - P_N^* T P_N\| < \vr$ for any $N \geq M$. Applying \eqref{eq:x_k} to
$u = T - P_N^* T P_N$, and invoking \eqref{eq:s_m}, we conclude that
$x_k(T) \leq x_{k-N} (T - P_N^* T P_N) \leq K_G \vr/\sqrt{k-N}$
for every $k > N$. Thus,
$$
\limsup_k \sqrt{k} x_k(T) \leq \lim_k \sqrt{k/(k-N)} K_G \vr = K_G \vr .
$$
As $\vr > 0$ is arbitrary, we conclude that $\lim \sqrt{k} x_k(T) = 0$.

To establish $\lim \sqrt{k} y_k(T) = 0$, recall that
$s_m(u+v) \leq s_m(u) + \|v\|$
for any operators $u$ and $v$, and any $s$-scale $(s_m)$. In particular,
for any $S \in B(c_0, \ell_1)$,
$y_k(S) \leq y_k(P_M^* S P_M) + \|S - P_M^* S P_M\|$ for any $M \in \N$.
By duality, $y_k(u) = x_k(u^*)$ if $u$ is an operator between finite dimensional
spaces. Viewing $P_M^* S P_M$ as an element of $B(\ell_\infty^M, \ell_1^M)$,
and applying \eqref{eq:x_k}, we obtain
$y_k(P_M^* S P_M) = x_k((P_M^* S P_M)^*) \leq K_G \|S\|/\sqrt{k}$.
As $\lim_M \|S - P_M^* S P_M\| = 0$, we conclude that $y_k(S) \leq K_G \|S\|/\sqrt{k}$.

Using the inequality from the previous paragraph with $S = T - P_N^* T P_N$
($T \in B(c_0, \ell_1)$, $k > N \in \N$), and invoking \eqref{eq:s_m}, we conclude
that $y_{k}(T) \leq K_G \|T - P_N^* T P_N\|/\sqrt{k-N}$. Applying
\eqref{eq:truncate} (as in the case of $x_k(T)$) yields $\lim_k \sqrt{k} y_k(T) = 0$.

Finally we tackle $(h_k(\cdot))$. Show first that, for any
$S  \in B(c_0, \ell_1)$, and any two contractions
$u : \ell_2 \to c_0$ and $v : \ell_1 \to \ell_2$,
we have $a_k(vSu) \leq K_G^2 \|S\|/k$. To achieve this,
denote the nuclear norm of an operator by $\nu( \cdot )$.
By 
\cite[Sections 1 and 5]{PiFACT},
$$
\nu(vSu) \leq \pi_2(v) \pi_2(S) \|u\| \leq K_G^2 \|v\| \|S\| \|u\| = K_G^2 \|S\| .
$$
Denoting the singular numbers of $vSu$
by $\lambda_1 \geq \lambda_2 \geq \ldots \geq 0$, we see that
$\nu(vSu) = \lambda_1 + \lambda_2 + \ldots \leq K_G^2 \|S\|$, and
$a_k(vSu) = \lambda_k$. Clearly, $\lambda_k \leq \nu(vSu)/k \leq K_G^2 \|S\|/k$.



Now consider $T \in B(c_0, \ell_1)$. For any $\vr > 0$, there exists $M \in \N$
such that $\|T - P_N^* T P_N\| < \vr$ for $N \geq M$. Combining the previous paragraph with
\eqref{eq:s_m}, we see that, for $k > N$,
$a_k(T) \leq a_{k-N}(T - P_N^* T P_N) \leq K_G^2 \vr/(k-N)$.
Thus, $\limsup_k k a_k(T) \leq K_G^2 \vr$. As $\vr > 0$ is arbitrary, 
the proof is complete.
\end{proof}

\begin{proof}[Proof of Proposition~\ref{prop:second}]
Note first that, if $Y$ is an $L_1$ predual, and $T : X \to Y$ is
a compact operator, then $c_k(T) = a_k(T)$ for any $k$.
Indeed, fix $\vr > 0$, and find $E \hra X$ such that $\dim X/E < k$,
and $\|T|_E\| < c_k(T) + \vr/2$. By \cite{Li}, there exists $S : X \to Y$
such that $S|_E = T|_E$, and $\|S\| < c_k(T) + \vr$. Let $u = T - S$.
Then $\rank u < k$, and $a_k(T) \leq \|T - u\| = \|S\| < c_k(T) + \vr$.
As $\vr$ is arbitrary, we are done.

Thus, it suffices to show the non-existence of a $T \in K(X,c_0)$
with $\|T\| = c_2(T) = 1$. Suppose, for the sake of contradiction,
that such a $T$ exists. Then there exists a unique sequence
$(x^*_i)_{i \in \N} \in c_0(X^*)$ such that $\max_i \|x^*_i\| = \|T\| = 1$,
and $T x = (\langle x, x^*_i \rangle)_{i \in \N}$ for every
$x \in X$. Let $N = \max\{i : \|x^*_i\| = 1\}$.
If $c_2(T) = 1$, then, for every $1$-codimensional $E \hra X$,
$$
\max_{1 \leq i \leq N} \sup_{x \in E, \|x\| \leq 1}
|\langle x_i^* , x \rangle | = 1.
$$
By the reflexivity of $X$, the $\sup$ in the centered expression is
attained. Therefore, for every such $E$ there exists $i \in \{1, \ldots, N\}$
such that $E$ contains $E_i$, where $E_i$ is the (one-dimensional) linear
span of the unique $x_i \in X$ satisfying $\|x_i\| = 1 = \langle x_i^*, x_i \rangle$.
In other words, any $x^* \in X^*$ satisfies $\langle x^*, x_i \rangle = 0$,
for some $i$. This, however, is impossible.
\end{proof}

\begin{proof}[Proof of Theorem~\ref{cor:factor}]
We re-use the operator $T$ constructed in the
proof of Theorem~\ref{thm:controlled}, and the notation introduced there.
The desired lower estimate follows from $c^{(\Gamma_p)}_m(T) \geq c_m(T)$.
To estimate $a_m^{(\Gamma_p)}(T)$ from above,
assume first $m > 3 n_k^\prime/2$, where $n^\prime_k = n_1 + \ldots + n_{k-1}$.
Write $T = T^{(1)} + T^{(2)} + T^{(3)}$, where
$$
T^{(1)} = \sum_{s=1}^{k-1} j_s S_s q_s , \, \, \,
T^{(2)} = j_k S_k q_k , \, \, \, {\mathrm{and}} \, \, \,
T^{(3)} = \sum_{s=k+1}^\infty j_s S_s q_s .
$$
Then
$$
a_m^{(\Gamma_p)}(T) \leq
a_m^{(\Gamma_p)}(T^{(1)} + T^{(2)}) + \gamma_p(T^{(3)}) \leq
a_{m - \rank T^{(1)}}^{(\Gamma_p)}(T^{(2)}) +
\sum_{s=k}^\infty \gamma_p(D_{s+1}) .
$$
Note that $L_p$ contains a $C_p$-complemented copy of $L_2$ (with
$C_p \sim \max\{\sqrt{p}, 1/\sqrt{p-1}\}$), hence
$\gamma_p(D_{s+1}) \leq C_p \|D_{s+1}\| \leq C_p \alpha_{n_s+1}$.
But, for $j \geq 0$,
$\alpha_{n_{k+j}+1} \leq 5^{-j} \alpha_{n_k+1} \leq 5^{-j} \alpha_m$,
hence $\sum_{s=k}^\infty \gamma_p(D_{s+1}) \leq 5 C_p \alpha_m/4$.
Furthermore, $\rank T^{(1)} \leq n^\prime_k$.
Therefore, 
$$
a_{m - n^\prime_k}^{(\Gamma_p)}(T^{(2)}) \leq a_{m - n^\prime_k}^{(\Gamma_p)}(D_k) \leq
C_p \beta_{m - n^\prime_k,k} = C_p \alpha_{\max\{m - n^\prime_k, n_{k-1}+1\}} \leq
C_p \alpha_{\lceil m/2 \rceil} ,
$$
hence $a_m^{(\Gamma_p)}(T) \leq 3 C_p \alpha_{\lceil m/3 \rceil}$.

Finally, we handle the case of $n_{k-1} < m \leq 3 n_k^\prime/2$
as in Theorem~\ref{thm:controlled}.
\end{proof}

The proof of Theorem~\ref{thm:2sum}
requires a technical result, which may be known to experts.
We say that a sequence $(\alpha_k)_{k \in \N}$ is {\it convex} if
$$
\alpha_k \leq \frac{n-k}{n-m} \alpha_m + \frac{k-m}{n-m} \alpha_n
$$
whenever $m < k < n$. It is easy to see that, for any non-increasing convex sequence
$(\alpha_k)$ of non-negative numbers,
\begin{equation}
\frac{\alpha_i - \alpha_j}{j-i} \geq \frac{\alpha_m - \alpha_n}{n-m}
\label{eq:slopes}
\end{equation}
if $j > i$, $n > m$, $i \leq m$, and $j \leq n$.

\begin{lemma}\label{l:convex}
Suppose $(\alpha_k)_{k \in \N}$ is a non-increasing sequence, converging to $0$.
Then there exists a convex sequence $(\beta_k)_{k \in \N}$
satisfying $\alpha_k \geq \beta_k \geq \min\{\alpha_k/2, \alpha_{2k-1}\}$ for any $k$.
\end{lemma}

\begin{proof}
Set $\beta_1 = \alpha_1$. For $k > 1$, define 
$$
\beta_k = \inf_{m \leq k \leq n, m < n}
\Big\{ \frac{n-k}{n-m} \alpha_m + \frac{k-m}{n-m} \alpha_n \Big\} .
$$
The standard ``convex envelope'' arguments (see e.g. \cite[p.~66]{LT2}) show that
$(\beta_k)$ is indeed a convex sequence. Thus, it suffices to show that
\begin{equation}
\frac{n-k}{n-m} \alpha_m + \frac{k-m}{n-m} \alpha_n \geq \min\{\alpha_k/2, \alpha_{2k-1}\}
\label{eq:conv}
\end{equation}
if $m \leq k \leq n$ and $m < n$. 
If $n < 2k$, then
$$
\frac{n-k}{n-m} \alpha_m + \frac{k-m}{n-m} \alpha_n \geq 
\frac{n-k}{n-m} \alpha_n + \frac{k-m}{n-m} \alpha_n \geq \alpha_{2k-1} .
$$
On the other hand, if $n \geq 2k$, then $(n-k)/(n-m) > 1/2$, and
$$
\frac{n-k}{n-m} \alpha_m + \frac{k-m}{n-m} \alpha_n \geq 
\frac{n-k}{n-m} \alpha_m > \frac{\alpha_k}{2} .
$$
In either case, (\ref{eq:conv}) holds.
\end{proof}


We also need to be able to estimate $p$-summing norms of diagonal operators.

\begin{lemma}\label{l:p_sum_diag}
For $p \in [2,\infty)$, there exists $\kappa_p \in (0,1]$ ($\kappa_2 = 1$) such that:
\begin{enumerate}
\item
If $u$ is an operator on a Hilbert space, then
$\kappa_p \|u\|_{HS} \leq \pi_p(u) \leq \|u\|_{HS}$
\item
If $D = \diag(d_i)_{i=1}^N$ is a diagonal operator from $\ell_\infty^N$
to $\ell_2^N$, then
$\kappa_p (\sum_i |d_i|^2 )^{1/2} \leq \pi_p(D) \leq (\sum_i |d_i|^2 )^{1/2}$.
\end{enumerate}
\end{lemma}

\begin{proof}
Part (1) can be found in e.g.~\cite{DJT}. Part (2) is also known.
We provide the proof for the sake of completeness. By scaling, we can assume
that $\sum_i |d_i|^2 = 1$.

Consider the case of $p = 2$ first. Pietsch Factorization Theorem yields
$\pi_2(D) \leq 1$. On the other hand, let $(e_i)_{i=1}^N$ be the canonical basis
for $\ell_\infty^N$. Then
$$
\Big( \sum_{i=1}^N \|D e_i\|^2 \Big)^{1/2} = 1 =
\sup_{f \in \ell_1^N, \|f\| = 1}
\Big( \sum_{i=1}^N |\langle f, e_i \rangle|^2 \Big)^{1/2} ,
$$
hence $\pi_2(D) \geq 1$. Thus, $\pi_2(D) = (\sum_i |d_i|^2 )^{1/2}$.

Now suppose $p > 2$. Trivially, $\pi_p(D) \leq \pi_2(D) = (\sum_i |d_i|^2 )^{1/2}$.
To prove the opposite inequality, denote by $id$ the formal identity map
from $\ell_2^N$ to $\ell_\infty^N$. By Part (1),
$$
\pi_p(D) = \pi_p(D) \|id\| \geq \pi_p(D \circ id)
\geq \kappa_p \|D \circ id\|_{HS} = \kappa_p (\sum_i |d_i|^2 )^{1/2} .
$$
\end{proof}

\begin{proof}[Proof of Theorem~\ref{thm:2sum}]
By Lemma~\ref{l:convex}, it suffices to show that, for any convex sequence
$(\alpha_k)$ convergent to $0$, there exists $T \in B(X,Y)$ with
the property that, for every $m$,
$$
c_p \alpha_{9m} \leq \cs_m(T) \leq a_m^{(\Pi_2)}(T) \leq 3 \alpha_{\lceil m/2 \rceil} .
$$
Set $n_0 = 0$, and find a ``rapidly increasing'' sequence
$(n_k)$ with the property that, for any $k \in \N$,
$n_k > 5(n_{k-1}+1)$, and $\alpha_{n_k} \leq \alpha_{5(n_{k-1}+1)}/5$
(the first inequality follows from the second if $\alpha_i > 0$ for every $i$).

Fix $c \in (1,6/5)$, and find, for each $k$, $n_k$-dimensional
spaces $E_k \hra X$ and $F_k \hra Y$, whose Banach-Mazur distance to $\ell_2^{n_k}$
is less than $\sqrt{c}$.
As in Lemma~\ref{l:dvor}, select the $F_k$'s in such a way that there exist finite rank
operators $R_k \in B(Y)$, such that $\|R_k\| < 5/2$,
$R_k|_{F_k} = I_{F_k}$, and $R_k|_{F_s} = 0$ if $k \neq s$.
Find contractions $U_k : X \to \ell_2^{n_k}$ and
$V_k : \ell_2^{n_k} \to F_k$, for which $\|U_k^{-1}\|, \|V_k^{-1}\| < \sqrt{c}$.
Denote by $id$ the formal identity map from $\ell_2^{n_k}$ to $\ell_\infty^{n_k}$.
Then $id \circ U_k$ extends to a contraction $W_k : X \to \ell_\infty^{n_k}$.

For $1 \leq j \leq n_k$, let
$\beta_{jk} = \sqrt{\alpha_{j+2n_{k-1}}^2 - \alpha_{j+2n_{k-1}+1}^2}$.
As the sequence $(\alpha_j)$ is convex, $\beta_{1k} \geq \ldots \geq \beta_{n_kk}$.
Let $D_k = \diag(\beta_{jk})_{j=1}^{n_k}$ be a diagonal map
from $\ell_\infty^{n_k}$ to $\ell_2^{n_k}$.
Consider the map $T = \sum_{k=1}^\infty V_k D_k W_k$.
As
$\pi_2(D_k)^2 = \sum_{j=1}^{n_k} \beta_{jk}^2 \leq \alpha_{2 n_{k-1}+1}^2$,
the operator $T$ is $2$-summing. We shall show that $T$ has the desired properties.

First estimate $\cs_j(T)$ from below. To this end, find $k$ such that
$n_{k-1} < j \leq n_k$. Suppose $Z \hra Y$ has codimension less than $j$.
Then $H = U_k(E_k \cap Z)$ is a subspace of $\ell_2^{n_k}$ of codimension
less than $j$. As $5 \|U_k^{-1}\| \|V_k^{-1}\|/2 < 3$,
$$
\eqalign{
\pi_p(T|_Z) \geq \pi_p(T|_{E_k \cap Z})
&
\geq
\frac{2}{5} \pi_p(R_k T|_{E_k \cap Z}) \geq
\frac{2}{5} \pi_p(R_k V_k D_k|_H U_k)
\cr
&
\geq
\frac{1}{3} \pi_p(D_k|_H) = \frac{\kappa_p}{3} \|D_k|_H\|_{HS} 
}
$$
(here, we view $D_k$ as an operator on $\ell_2^{n_k}$, and $\kappa_p$ is
the constant from Lemma~\ref{l:p_sum_diag}).
Weyl's Minimax Principle implies
$$
\|D_k|_H\|_{HS}^2 \geq \sum_{i=j}^{n_k} \beta_{ik}^2 =
\alpha_{j + 2 n_{k-1}}^2 - \alpha_{2 n_{k-1} + n_k + 1}^2 .
$$

Let $m_k$ be the largest $m \leq n_k$ for which
$\alpha_{m + 2 n_{k-1}} \geq 1.1 \alpha_{n_k + 2 n_{k-1} + 1}$
(by construction, $m_k > n_{k-1}$). Consider three cases:
(i) $n_{k-1} < j \leq m_k$, (ii) $m_k < j \leq n_k$ and $m_k \geq n_k/3$,
and (iii) $m_k < j \leq n_k$ and $m_k < n_k/3$.

If $n_{k-1} < j \leq m_k$, we obtain
$$
\|D_k|_H\|_{HS} \geq \sqrt{1 - (10/11)^2} \, \alpha_{j + 2 n_{k-1}} \geq
\frac{1}{3} \alpha_{j + 2 n_{k-1}} \geq \frac{1}{3} \alpha_{3 j - 2} ,
$$
hence $\pi_p(T|_Z) \geq \kappa_p \alpha_{j + 2 n_{k-1}} \geq \kappa_p  \alpha_{3j-2}/9$.
As this inequality holds whenever $\dim Y/Z < j$, we conclude that
$\cs_j(T) \geq \kappa_p \alpha_{3j-2}/9$.

Now suppose $m_k < j \leq n_k$.
As $n_k < m_{k+1}$, the reasoning above yields
\begin{equation}
\cs_j(T) \geq \cs_{n_k+1}(T) \geq \kappa_p \alpha_{3 n_k + 1}/9 .
\label{eq:bdry}
\end{equation}
If $j > m_k \geq n_k/3$, we conclude that $\cs_j(T) \geq \kappa_p \alpha_{9j}/9$.

It remains to consider the case when $m_k \leq n_k/3$. Then
\eqref{eq:slopes} implies
$$
\frac{\alpha_{m_k + 2 n_{k-1} + 1} - \alpha_{n_k + 2 n_{k-1} + 1}}{n_k - m_k} \geq
\frac{\alpha_{n_k + 1} - \alpha_{3 n_k + 1}}{2 n_k} ,
$$
hence
$$
\eqalign{
\alpha_{n_k + 1} - \alpha_{3 n_k + 1}
&
\leq
\frac{2 n_k}{n_k - m_k}
\big( \alpha_{m_k + 2 n_{k-1} + 1} - \alpha_{n_k + 2 n_{k-1} + 1} \big)
\cr
&
\leq
\frac{2}{2/3} (1.1 - 1) \alpha_{n_k + 2 n_{k-1} + 1} 
\leq 0.3 \alpha_{n_k + 1} ,
}
$$
hence
\begin{equation}
\alpha_{3 n_k + 1} \geq 0.7 \alpha_{n_k + 1}
\label{eq:times3}
\end{equation}
Using (\ref{eq:bdry}), we obtain, for $j > m_k$,
$$
\eqalign{
\kappa_p^{-1} \cs_j(T)
&
\geq
\kappa_p^{-1} \cs_{n_k+1}(T)
\geq \frac{1}{9} \alpha_{3 n_k + 1} \geq
\frac{7}{90} \alpha_{n_k + 1}
\cr
&
\geq
\frac{7}{90} \alpha_{n_k + 2 n_{k-1} + 1} \geq
\frac{7}{90 \cdot 1.1} \alpha_{m_k + 2 n_{k-1} + 1} \geq 
\frac{7}{99} \alpha_{3j}
}
$$
(here, we use the fact that $m_k > n_{k-1}$).

Next we estimate $\at_j(T)$ from above.
Denote by $P_{sk}$ the projection onto the first $s$ coordinates of
$\ell_\infty^{n_k}$. Then
$$
\pi_2(D_k(I - P_{sk}))^2 = \sum_{j=s+1}^{n_k} \beta_{jk}^2 =
\alpha_{s + 2 n_{k-1}}^2 - \alpha_{n_k + 2 n_{k-1} + 1}^2 \leq
\alpha_{s + 2 n_{k-1}}^2 .
$$
If $n_1 + \ldots + n_{k-1} < j \leq n_k$, then
$$
u = \sum_{s<k} U_s D_s W_s +
U_k D_k P_{j - (1 + n_1 + \ldots + n_{k-1}), k} W_k .
$$
has rank less than $j$, hence
$$
\eqalign{
\at_j(T)
&
\leq
\pi_2(T - u) \leq \sum_{s > k} \|V_s\| \pi_2(D_s) \|W_s\| +
\|U_k\| \pi_2(D_k(I - P_{j - (1 + n_1 + \ldots + n_{k-1}), k})) \|W_k\|
\cr
&
\leq
\sum_{s \geq k} \alpha_{2 n_s+1} +
\alpha_{j+n_{k-1}-(1+n_1+\ldots+n_{k-2})}
\leq 3 \alpha_j
}
$$
(here, we use the fact that $n_{k-1} > 2 (1 + n_1 + \ldots + n_{k-2})$, and
$\alpha_{n_{s+1}} \leq \alpha_{5(n_s+1)}/5$, for each $s$).
If $n_{k-1} < j \leq n_1 + \ldots + n_{k-1}$, then,
by the above reasoning,
$$
\at_j(T) \leq \at_{n_{k-1}}(T) \leq 3 \alpha_{n_{k-1}} \leq
3 \alpha_{\lceil{4j/5\rceil}} ,
$$
since $n_{k-1} > 4(n_1 + \ldots + n_{k-1})/5$.
\end{proof}

To establish Theorems~\ref{thm:no_cotype} and \ref{thm:type},
we need to prove two lemmas.

\begin{lemma}\label{l:l_infty}
Suppose $X$ is a Banach space without non-trivial cotype,
and $E$ and $X^\prime$ are subspaces of $X$ of finite dimension and codimension,
respectively. Then, for every $n \in \N$ and $\vr > 0$, there exists a $n$-dimensional
subspace $F \hra X^\prime$, such that $d(F, \ell_\infty^n) < 1 + \vr$, and
there exists a projection $P$ from $X$ onto $F$, such that $\|P\| < 1 + \vr$, and
$P|_E = 0$.
\end{lemma}

\begin{lemma}\label{l:l_2}
Suppose a Banach space $X$ has Property $\pp_C$, and
$E$ and $X^\prime$ are subspaces of $X$ of finite dimension and codimension,
respectively. Then, for every $n \in \N$ and $\vr > 0$, there exists a $n$-dimensional
subspace $F \hra X^\prime$, such that $d(F, \ell_2^n) \leq C$, and there exists
a projection $P$ from $X$ onto $F$, such that $\|P\| < C^2 + \vr$, and $P|_E = 0$.
\end{lemma}

To establish these two lemmas, we need a ``small perturbation'' result.

\begin{lemma}\label{l:small_pert}
Suppose $E$ and $F$ are subspaces of a Banach space $X$, with $\dim F = n < \infty$,
and $P$ is a projection from $X$ onto $F$, with $\|P|_E\| < \vr$ ($0 < \vr < 1/8$).
Then there exists a projection $Q$ from $X$ onto $F$, such that $Q|_E = 0$, and
$\|P - Q\| \leq 4 \|P\| n \vr$.
\end{lemma}

\begin{proof}
Note first that, for any $e \in E$ and $f \in F$,
\begin{equation}
\|e+f\| \geq (\|e\| + \|f\|)/(4\|P\|)
\label{eq:dir_sum}
\end{equation}
Indeed,
$$
\|P\| \|e+f\| \geq \|P(e+f)\| \geq \|f\| - \|Pe\| \geq \|f\| - \vr \|e\| .
$$
Moreover,
$$
(1+\|P\|) \|e+f\| \geq \|I-P\| \|e+f\| \geq \|(I-P)(e+f)\| \geq (1 - \vr) \|e\| .
$$
Therefore,
$$
\eqalign{
\|e+f\|
&
=
\frac{\|P\|}{2\|P\|+1} \|e+f\| + \frac{\|P\|+1}{2\|P\|+1} \|e+f\|
\cr
&
\geq
\frac{1}{2\|P\|+1} \big( \|f\| - \vr \|e\| + (1 - \vr) \|e\| \big) =
\frac{\|f\| + (1-2\vr)\|e\|}{2\|P\|+1} \geq
\frac{3}{4} \cdot \frac{\|f\| + \|e\|}{3 \|P\|} ,
}
$$
yielding (\ref{eq:dir_sum}).

Fix an Auerbach basis $(f_i)_{i=1}^n$ in $F$. Then there exist norm $1$
elements $f_i^* \in F^*$ satisfying $\langle f_i^*, f_j \rangle = \delta_{ij}$
(Kronecker's delta). Let $x_i^* = P^* f_i^*$. Then $\|x_i^*\| \leq \|P\|$
($1 \leq i \leq n$), and, for every $x \in X$,
$Px = \sum_{i=1}^n \langle f_i^*, Px \rangle f_i =
\sum_{i=1}^n \langle x_i^*, x \rangle f_i$.
Therefore, $\|x_i^*|_E\| \leq \|P|_E\| < \vr$.
Define $y_i^* \in (E+F)^*$ by setting $x_i^*|_E = y_i^*|_E$, and
$y_i^*|_F = 0$. By (\ref{eq:dir_sum}), $\|y_i^*\| \leq 4 \|P\| \vr$.
By Hahn-Banach Theorem, there exist $z_i^* \in X^*$ ($1 \leq i \leq n$)
such that $z_i^*|_E = x_i^*|_E$, $z_i^*|_F = 0$, and $\|z_i^*\| \leq 4 \|P\| \vr$.
Then the projection $Q$, defined by
$Qx = \sum_{i=1}^n \langle x_i^* - z_i^* , x \rangle f_i$,
has the desired properties.
\end{proof}

\begin{proof}[Proof of Lemma~\ref{l:l_infty}]
Fix $n \in \N$ and $\delta \in (0, 1/8)$.
By compactness, there exists $M \in \N$ such that, for every collection
$(z_s)_{s=1}^M$ in $\ball(E^*)$, there exist $n$ pairs
$$
(p_i, q_i) \in \{1, \ldots, M\}^2 \backslash \{(1,1), \ldots, (M,M)\}
\, \, (1 \leq i \leq n) ,
$$
such that $\{p_i,q_i\} \cap \{p_j, q_j\} = \emptyset$ unless $i=j$,
and $\|z_{p_i} - z_{q_i}\| < \delta/n$ for every $i$.

By Krivine-Maurey-Pisier Theorem (see e.g.~\cite{Ma}), for every $\delta > 0$
there exists $G \hra X^\prime$ with $d(G, \ell_\infty^M) < 1 + \delta$.
Find a contraction $U : G \to \ell_\infty^M$ such that $\|U^{-1}\| < 1+\delta$, and
extend it to a contraction $\tilde{U} : X \to \ell_\infty^M$.
There exist $(x_i^*)_{i=1}^M$ in the unit ball of $X^*$ such that
$\tilde{U} x = \sum _{i=1}^M\langle x_i^*, x \rangle \sigma_i$, where
$(\sigma_i)$ is the canonical basis on $\ell_\infty^M$
(hence, $\langle x_i^*, U^{-1} \sigma_j \rangle$ equals
$1$ if $i=j$, $0$ otherwise).

By our choice of $M$, there exist disjoint pairs $(p_i,q_i)$ ($1 \leq i \leq n$)
such that, for each $i$, $\|(x^*_{p_i} - x^*_{q_i})|_E\| < \delta/n$.
Let $\tilde{F} = \span[ \sigma_{p_i} - \sigma_{q_i} : 1 \leq i \leq n] \hra \ell_\infty^M$,
and $F = U^{-1} (\tilde{F})$. Then $\tilde{F}$ is isometric to $\ell_\infty^n$,
and $d(F, \ell_\infty^n) < 1 + \delta$. Furthermore, $\tilde{F}$ is the
range of the contractive projection $Q$, defined by setting
$$
Q \sigma_j = \left\{ \begin{array}{ll}
   0                                   &   j \notin \cup_i \{p_i, q_i\}   \cr
   (\sigma_{p_i} - \sigma_{q_i})/2     &   j = p_i   \cr
   - (\sigma_{p_i} - \sigma_{q_i})/2   &   j = q_i
\end{array} \right. .
$$
Then $P = U^{-1} Q \tilde{U}$ is a projection onto $F$, with $\|P\| < 1 + \delta$.
Moreover,
$$
Px = \frac{1}{2} \sum_{i=1}^n \langle x_{p_i}^* - x_{q_i}^* , x \rangle
U^{-1} (\sigma_{p_i} - \sigma_{q_i})
$$
for $x \in X$, hence $\|P|_E\| < \delta$. As $\delta > 0$ can be chosen to be
arbitrarily small, an application of Lemma~\ref{l:small_pert} completes the proof.
\end{proof}

\begin{proof}[Proof of Lemma~\ref{l:l_2}]
Fix $n \in \N$ and $\delta \in (0, 1/8)$.
Select a $\delta C^{-2}$-net $(e_i)_{i=1}^M$ in $\ball(E)$.
Pick $m > M C^4/\delta^2$. Find $G \hra X^\prime$, which is $C$-isomorphic to $\ell_2^{mn}$,
and $C$-complemented in $X$. Consider a contraction $U : G \to \ell_2^{mn}$ such that
$\|U^{-1}\| \leq C$, and a projection $P$ from $X$ onto $G$, with $\|P\| \leq C$.
Denote by $(\sigma_j)_{j=1}^{mn}$ the canonical basis for $\ell_2^{nm}$, and
let $Q_k$ ($1 \leq k \leq m$) be the orthogonal projection from $\ell_2^{nm}$ onto
$\span[\sigma_i : (k-1)n + 1 \leq i \leq kn]$. Then, for every $k$,
$P_k = U^{-1} Q_k U P$ is a projection of norm not exceeding $C^2$, whose range is
$C$-isomorphic to $\ell_2^n$. We claim that there exists $k$ such that
$\|P_k e\| < 2 \delta$ for any $e \in \ball(E)$. Once the existence of
such $k$ is established, we can complete the proof by applying
Lemma~\ref{l:small_pert}  to $P_k$ and
$F_k = \ran P_k = U^{-1} \span[\sigma_i : (k-1)n + 1 \leq i \leq kn]$.

Note that, if $x_k \in \ran P_k$ for $1 \leq k \leq m$, then
$$
\|\sum_{k=1}^m x_k\| \geq \|\sum_k U x_k\| =
\Big( \sum_{k=1}^m \|U x_k\|^2 \Big)^{1/2} \geq
C^{-1} \Big( \sum_{k=1}^m \|x_k\|^2 \Big)^{1/2} .
$$
Thus, for any $x \in \ball(X)$,
$C^4 \geq C^2 \|P x\|^2 \geq \sum_{k=1}^m \|P_k x\|^2$,
hence $\|P_k x\| \geq \delta$ for at most $C^4/\delta^2$ values of $k$.
As $m > M C^4/\delta^2$, there exists $k$ such that $\|P_k e_i\| < \delta$
for every $i \in \{1, \ldots, M\}$. For every $e \in \ball(E)$,
find $i$ such that $\|e - e_i\| < \delta/(4 C^2)$. Then
$$
\|P_k e\| \leq \|P_k e_i\| + \|P_k\| \|e - e_i\| <
\delta + C^2 \cdot \frac{\delta}{C^2} < 2 \delta ,
$$
as desired.
\end{proof}

\begin{proof}[Proof of Theorem~\ref{thm:no_cotype}]
By Lemma~\ref{l:convex}, it suffices to show that, for any
convex sequence $\alpha_i \searrow 0$, there exists $T \in B(X,Y)$ satisfying
$$
4 \alpha_{\lceil{4m/5}\rceil} \geq \asa_m(T) \geq c_m(T) \geq \frac{49}{1100} \alpha_{9m} .
$$

Find a sequence $0 = n_0 < n_1 < n_2 < \ldots$ such that, for each $k$,
$\alpha_{n_k} \leq \alpha_{5(n_{k-1}+1)}/5$, and $n_k > 5 (n_{k-1} + 1)$.
Find sequences of subspaces $E_k  \hra X$ and $F_k \hra Y$ in such a way that:
\begin{enumerate}
\item
There exist contractions $U_k : E_k \to \ell_\infty^{n_k}$ and
$V_k : \ell_1^{n_k} \to F_k$ such that their inverses have norms less than $2^{1/4}$.
\item
For each $k$, there exists a projection $P_k$ onto $E_k$ such that $\|P_k\| < 2^{1/4}$,
and $P_k|_{E_j} = 0$ for $j \neq k$ (in other words, $P_j P_k = 0$ if $j \neq k$).
\end{enumerate}
The existence of $(F_k)$ follows from the fact that $Y$ has no non-trivial type
\cite{Ma}. Select $(E_k)$ inductively. Select $E_1$ to be
arbitrary, subject to the estimate on $d(E_1, \ell_\infty^{n_1})$. Now suppose
$E_1, \ldots, E_{k-1}, P_1, \ldots, P_{k-1}$ have already been defined.
By Lemma~\ref{l:l_infty}, there exists $E_k \hra \cap_{j=1}^{k-1} \ker P_j$,
and a projection $P_k$ onto it, such that $d(E_k, \ell_\infty^{n_k}) < 2^{1/4}$,
$\|P_k\| < 2^{1/4}$, and $P_k|_{E_j} = 0$ for any $j < k$.

For $1 \leq i \leq n_k$, set
$\beta_{ik} = \alpha_{i + 2 n_{k-1}} - \alpha_{i + 2 n_{k-1} + 1}$.
By the convexity of $(\alpha_i)$,
$\beta_{1k} \geq \beta_{2k} \geq \ldots \geq \beta_{n_k k}$.
Let $D_k = \diag(\beta_{ik})$ be the diagonal map from
$\ell_\infty^{n_k}$ to $\ell_1^{n_k}$, and set $S_k = V_k D_k U_k$
(we can view $S_k$ as a map into $Y$).
We claim that the operator $T = \sum_j S_j P_j$ has the desired properties.
To this end, recall that (see e.g.~\cite[Section 11.11]{PieID}),
for a diagonal operator $D = \diag(d_i) \in B(\ell_\infty^n, \ell_1^n)$,
$a_m(D) = c_m(D) = \sum_{i=m}^n d_i$ (here, we are assuming that
$m \leq n$, and $d_1 \geq d_2 \geq \ldots \geq d_n$). Furthermore,
for any ideal $\A$, $\|D\|_{\A} \leq \sum_i d_i$ (to see this, represent $D$
as a sum of rank $1$ diagonal operators), hence
$c_m(D) = \asa_m(D) = \sum_{i=m}^n d_i$.

To estimate $c_m(T)$ from below, find $k$ such that
$n_{k-1} < m \leq n_k$. By the injectivity of $\A$,
$$
c_m(T) \geq c_m(T|_{E_k}) = c_m(S_k) \geq 2^{-1/2} \csa_m(D_k) \geq 2^{-1/2} c_m(D_k) .
$$
As noted above, 
$c_m(D_k) = \sum_{i=m}^{n_k} \beta_{ik} =
\alpha_{m+2n_{k-1}} - \alpha_{n_k + 2 n_{k-1} + 1}$.
As in the proof of Theorem~\ref{thm:2sum},
let $m_k$ be the largest number $m \leq n_k$ for which
$\alpha_{m+2n_{k-1}} \geq 1.1 \alpha_{n_k + 2 n_{k-1} + 1}$.
If $m \leq m_k$, then $c_m(D_k) \geq 0.1 \alpha_{m+2n_{k-1}} \geq 0.1 \alpha_{3m}$,
hence $c_m(T) \geq 0.07 \alpha_{3m}$. For $m > m_k$, recall that
$n_k+1 \leq m_{k+1}$, hence
$c_m(T) \geq c_{n_k+1}(T) \geq 0.07 \alpha_{3n_k+1}$.
If $m > n_k/3$, this yields $c_m(T) \geq 0.07 \alpha_{9m}$.
If $m_k < m \leq n_k/3$, \eqref{eq:times3} implies
$\alpha_{3 n_k + 1} \geq 0.7 \alpha_{n_k + 1}$.
Therefore,
$$
c_m(T) \geq c_{n_k+1}(T) \geq 0.07 \alpha_{3n_k+1} \geq
\frac{7^2}{1000} \alpha_{n_k + 2 n_{k-1} + 1} \geq
\frac{49}{1100} \alpha_{m_k + 2 n_{k-1} + 1} \geq \frac{49}{1100} \alpha_{3m} .
$$

Next estimate $\asa_m(T)$ from above. Suppose $n_1 + \ldots + n_{k-1} < m \leq n_k$.
Then
$$
\asa_m(T) \leq \asa_{m-(n_1+\ldots+n_{k-1})}(S_k P_k) +
\sum_{j=k}^\infty \|S_{j+1} P_{j+1}\|_{\mathcal{A}} \leq
2^{3/4} \big( \asa_\ell(D_k) + \sum_{j=k}^\infty \|D_{j+1}\|_\A \big) ,
$$
where $\ell = m-(n_1+\ldots+n_{k-1})$.
But $\|D_{j+1}\|_\A \leq \alpha_{n_j + 2 n_{j-1} + 1} \leq \alpha_{n_j + 1}$.
Moreover, $\alpha_{n_{s+1}+1} \leq \alpha_{n_s+1}/5$ for any $s$, hence
$\sum_{j=k+1}^\infty \|D_j\|_\A \leq 5 \alpha_{n_k+1}/4 \leq 5 \alpha_m/4$.
As noted previously,
$$
\asa_\ell(D_k) = \sum_{i=\ell}^{n_k} \beta_{i n_k} \leq
\alpha_{m + n_{k-1} - n_{k-2} - \ldots - n_1} \leq \alpha_m .
$$
Therefore, $\asa_m(T) \leq 2^{3/4} ( \alpha_m + 5 \alpha_m/4) \leq 4 \alpha_m$.

Now suppose $n_{k-1} < m \leq n_1 + \ldots + n_{k-1}$. Then
$\asa_m(T) \leq \asa_{n_{k-1}+1}(T) \leq 4 \alpha_{n_{k-1}+1}$.
As $m \leq 5 n_{k-1}/4$, we conclude that
$\asa_m(T) \leq 4 \alpha_{\lceil 4m/5 \rceil}$.
\end{proof}

\begin{proof}[Proof of Theorem~\ref{thm:type}]
The proof is very similar to the proof of Theorem~\ref{thm:no_cotype}.
Suppose $X$ has Property $\pp_C$.
By Lemma~\ref{l:convex}, it suffices to show that, for any convex sequence
$\alpha_i \searrow 0$, there exists an operator $T \in B(X,Y)$ such that
$$
\frac{1}{10C^2} \alpha_{9m} \leq \css_m(T) \leq \ass_m(T) \leq
4 \alpha_{\lceil 4m/5 \rceil}
$$
To this end, pick $C_1 \in (C^2, 70 C^2/66)$.
Find a sequence $0 = n_0 < n_1 < n_2 < \ldots$ such that, for each $k$,
$\alpha_{n_k} \leq \alpha_{5(n_{k-1}+1)}/5$, and $n_k > 5 (n_{k-1} + 1)$.
Find sequences of subspaces $E_k \hra X$ and $F_k \hra Y$ in such a way that:
\begin{enumerate}
\item
There exist contractions $U_k : E_k \to \ell_2^{n_k}$ and
$V_k : \ell_2^{n_k} \to F_k$, such that $\|U_k^{-1}\| \leq C$, and $\|V_k^{-1}\| < 2$.
\item
For each $k$, there exists a projection $P_k$ onto $E_k$ such that $\|P_k\| \leq C_1$,
and $P_k P_j = 0$ for $k \neq j$. 
\end{enumerate}
The existence of $(F_k)$ follows from Dvoretzky's theorem.
Select $(E_k)$ inductively. Pick an arbitrary $E_1$, satisfying
$d(E_1, \ell_2^{n_1}) \leq C_1$.Now suppose
$E_1, \ldots, E_{k-1}, P_1, \ldots, P_{k-1}$ have already been defined.
By Lemma~\ref{l:l_2}, there exists $E_k \hra \cap_{j=1}^{k-1} \ker P_j$,
and a projection $P_k$ onto it, such that $d(E_k, \ell_2^{n_k}) \leq C_1$,
$\|P_k\| \leq C_1$, and $P_k|_{E_j} = 0$ for any $j < k$.

For $1 \leq i \leq n_k$, let
$\beta_{ik} = \big(\alpha_{i+2n_{k-1}}^q - \alpha_{n_k+2n_{k-1}+1}^q\big)^{1/q}$,
where $1/q = 1/2 - 1/r + 1/t$.
By convexity, $\beta_{1k} \geq \beta_{2k} \geq \ldots \geq \beta_{n_k k}$.
Let $D_k = \diag(\beta_{ik})$ be the diagonal map on $\ell_2^{n_k}$, and set
$S_k = V_k D_k U_k$ (we can view $S_k$ as a map into $Y$).
We claim that the operator $T = \sum_j S_j P_j$ has the desired properties.

We rely on a result of Mitiagin \cite[Theorem 11.9]{T-J}:
for an operator $u$ on a Hilbert space,
$\|u\|_q \leq \pi_{t,r}(u) \leq {\mathfrak{a}}^{-1} \|u\|_q$,
where ${\mathfrak{a}} = \sqrt{2/\pi}$
is the first absolute Gaussian moment.

First estimate $\css_m(T)$ from below. For a fixed $m$, find $k$ such that
$n_{k-1} < m \leq n_k$. By the injectivity of $\Pi_{t,r}$,
$\css_m(T) \geq \css_m(T|_{E_k}) = \css_m(S_k) \geq \css_m(D_k)/(2C_1)$, and
$$
\css_m(D_k) = \inf_{\codim H < m} \pi_{tr}(D_k|_H) \geq
\inf_{\codim H < m} \|D_k|_H\|_q
$$
By \cite{GK},
$$
\inf_{\codim H < m} \|D_k|_H\|_q^q = \|\diag(\beta_{ik})_{i=m}^{n_k}\|_q^q =
\alpha_{m+2n_{k-1}}^q - \alpha_{n_k+2n_{k-1}+1}^q
$$
Let $m_k$ be the largest value of $m \leq n_k$ for which
$\alpha_{m + 2 n_{k-1}} \geq 1.1 \alpha_{n_k + 2 n_{k-1} + 1}$.
Emulate the proof of Theorems~\ref{thm:2sum}.
More precisely: if $n_{k-1} < m \leq m_k$, we have
$$
\css_m(D_k) \geq \big( 1 - (10/11)^q \big)^{1/q} \alpha_{m + 2 n_{k-1} + 1} \geq
\alpha_{3m}/3 ,
$$
hence $\css_m(T) \geq \alpha_{3m}/(6 C_1)$.
If $m > n_k/3$, we conclude that
$$
\css_m(T) \geq \css_{n_k+1}(T) \geq \alpha_{3 n_k + 1}/(6 C_1) \geq \alpha_{9m}/(6 C_1) .
$$
If $m_k < m \leq n_k/3$, \eqref{eq:times3} yields
$\alpha_{3 n_k + 1} > 0.7 \alpha_{n_k + 1}$, and therefore,
$$
6 C_1 \css_m(T) \geq 6 C_1 \css_{n_k+1}(T) \geq \alpha_{3 n_k + 1} \geq
\frac{7}{10} \alpha_{n_k + 1} \geq \frac{7}{10 \cdot 1.1} \alpha_{m_k + 2 n_{k-1} + 1} \geq
\frac{7}{11} \alpha_{3m} .
$$

Next estimate $\ass_m(T)$ from above. If
$n_1 + \ldots + n_{k-1} < m \leq n_k$, we obtain
$$
\ass_m(T) \leq \ass_{m-(n_1+\ldots+n_{k-1})}(S_k P_k) +
\sum_{j=k+1}^\infty \pi_{tr}(S_j P_j) \leq
C_1 \big( \ass_\ell(D_k) + \sum_{j=k+1}^\infty \pi_{tr}(D_j) \big) ,
$$
where $\ell = m-(n_1+\ldots+n_{k-1})$.
But, for $j \geq k$,
$$
{\mathfrak{a}} \pi_{tr}(D_{j+1}) \leq \|D_{j+1}\|_q \leq
\alpha_{n_j + 2 n_{j-1} + 1} \leq \alpha_{n_j + 1} \leq 5^{k-j} \alpha_m.
$$
Furthermore,
$$
\inf_{\rank u < \ell} \|D_k - u\|_q^q =
\| \diag_{i=\ell}^{n_k} (\beta_{ik})\|_q^q =
\sum_{i=\ell}^{n_k} \beta_{ik}^q \leq
\alpha_{\ell+2 n_k}^q \leq \alpha_m^q ,
$$
hence $\ass_m(D_k) \leq {\mathfrak{a}}^{-1} \alpha_m$. Therefore,
$\ass_m(T) \leq {\mathfrak{a}}^{-1} \alpha_m(1 + \sum_{s=0}^\infty 5^{-s})
\leq 4 \alpha_m$.

For $n_{k-1} < m \leq n_1 + \ldots + n_{k-1}$, we have
$\ass_m(T) \leq \ass_{n_{k-1}}(T) \leq 4 \alpha_{n_{k-1}} \leq
4 \alpha_{\lceil{4m/5}\rceil}$.
\end{proof}

{\noindent \bf Acknowledgments}.
I would like to thank A.~Aksoy and J.~M.~Almira for introducing me to the topic of ``lethargy,''
and for many stimulating conversations.


\end{document}